\documentclass[a4paper]{article}

\usepackage{amsmath}
\usepackage{amssymb}
\usepackage[pdftex,bookmarksnumbered=true]{hyperref}
\usepackage[pdftex]{graphicx,color} 

\newtheorem{lemma}[subsection]{Lemma}

\newtheorem{theorem}[subsection]{Theorem}
\newtheorem{corollary}[subsection]{Corollary}

\newtheorem{fact}[subsection]{Fact}
   {\refstepcounter{subsection}%
        \medbreak\noindent{\bf Question \thesubsection\space}}%
   {\par\medbreak}%
\newenvironment{remark}%
   {\refstepcounter{subsection}%
        \medbreak\noindent{\bf Remark \thesubsection\space}}%
   {\par\medbreak}%
   {\refstepcounter{subsection}%
        \medbreak\noindent{\bf Example \thesubsection\space}}%
   {\par\medbreak}%
\newenvironment{proof}%
   {\medbreak\noindent{\it Proof:\space}}%
   {\par\noindent\vrule height 5pt width 5pt depth 0pt\smallbreak}%
\newcommand{\df}{\em}
\let\sauvegardetiret=\-
\renewcommand{\-}[1]{\ifx#1-\penalty10000\hbox{-\relax}\penalty10000\else\sauvegardetiret#1\fi}
\newcommand{\tq}{\,\big/\ }
\newcommand{\NN}{{\mathbf N}}
\newcommand{\ZZ}{{\mathbf Z}}
\newcommand{\QQ}{{\mathbf Q}}
\newcommand{\cA}{{\cal A}}
\newcommand{\cB}{{\cal B}}
\newcommand{\cC}{{\cal C}}

\newcommand{\cF}{{\cal F}}
\newcommand{\cH}{{\cal H}}

\newcommand{\cL}{{\cal L}}
\newcommand{\cO}{{\cal O}}
\newcommand{\cU}{{\cal U}}

\newcommand{\cZ}{{\cal Z}}

\renewcommand{\tq}{\mathop{:}}

\newcommand{\UU}{{\mathbf U}}

\newcommand{\coalg}{\mathop{\rm coalg}}

\newcommand{\lland}{\mathop{\land\mskip-6mu\relax\land}}

\title{Cell decomposition and classification of definable sets in
$p$-optimal fields}
\author{Luck Darni\`ere\footnote{LAREMA, Facult\'e des sciences, 2 bd.
Lavoisier,  49045 Angers Cedex 01, France.}, Immanuel
Halpuczok\footnote{School of Mathematics, University of Leeds,
Woodhouse Lane, Leeds, UK.}
}

\begin{document}

\maketitle

\begin{abstract}
  We prove that for $p$\--optimal fields (a very large subclass of
  $p$\--minimal fields containing all the known examples) a cell
  decomposition theorem follows from methods going back to Denef's
  paper \cite{dene-1984}. We derive from it the existence of definable
  Skolem functions and strong $p$\--minimality. Then we turn to
  strongly $p$\--minimal fields satisfying the Extreme
  Value Property -- a property which in particular holds in fields
  which are elementarily equivalent to a $p$-adic one. For such fields
  $K$, we prove that every definable subset of $K\times K^d$ whose fibers
  over $K$ are inverse images by the valuation of subsets of
  the value group, are semi-algebraic. Combining the two we get a
  preparation theorem for definable functions on $p$\--optimal fields
  satisfying the Extreme Value Property, from which it follows that
  infinite sets definable over such fields are in definable
  bijection iff they have the same dimension. 
\end{abstract}

\section{Introduction}
\label{se:intro}

This paper is an attempt to continue the road opened by Haskell and
Macpherson in \cite{hask-macp-1997} toward a $p$\--adic version of
$o$\--minimality, by isolating large subclasses of $p$\--minimal
fields to which Denef's methods of \cite{dene-1984} apply
with striking efficiency. 

Recall that a {\df $p$\--adically closed} field is a field $K$
elementarily equivalent in the language of rings to a {\df
$p$\--adic field}, that is a finite extension of the field $\QQ_p$ of
$p$\--adic numbers. For every $a$ in $K$, $v(a)$
and $|a|$ denote the $p$\--valuation of $a$ and its norm. The norm is
nothing but the valuation with a multiplicative notation so that
$|0|=0$, $|ab|=|a|\cdot|b|$, $|a+b|\leq\max(|a|,|b|)$ and of course $|a|\leq|b|$
if and only if $v(a) \geq v(b)$. The valuation ring of $v$ is denoted by
$R$, and we fix some $\pi$ in $R$ such that $\pi R$ is the maximal ideal
of $R$. We let $v(K)$ or $|K|$ denote the image of $K$ by the
valuation. 

Throughout all this paper we consider a fixed expansion $(K,\cL)$ of a
$p$\--adically closed field $K$, that is an $\cL$\--structure
extending the ring structure of $K$ for some language $\cL$ containing
the language of rings. Except if otherwise specified, when we say that
a set or a function is definable we always mean ``definable in $\cL$
with parameters in $K$''. For sets and functions definable in the
language of rings (with parameters in $K$ as always), we use the term
``semi-algebraic'' instead. Wherever it is convenient we will identify
subsets of $K^m\times|K|^d$ with their inverse image in $K^{m+d}$ by the
valuation, thus saying for example that the former are definable,
semi-algebraic, and so on if the latter are so.

$(K,\cL)$ is {\df $p$\--minimal} if every definable subset of
$K$ is definable in the language of rings. It is {\df strongly
$p$\--minimal} (or $P$\--minimal for short, as in
\cite{hask-macp-1997}) if every elementarily equivalent
$\cL$\--structure is $p$\--minimal. When the distinction between the
$\cL$\--structure and the ring structure of $K$ is clear from the
context, $K$ itself is called a strongly $p$\--minimal field.

Strong $p$\--minimality was introduced by Haskell and Macpherson in
\cite{hask-macp-1997}. Since their proofs make extensive use of the
model-theoretic Compactness Theorem, very little is known on
$p$\--minimal fields without the ``strong'' assumption contrary to the
situation in $o$\--minimal expansions of real closed fields, where
$o$\--minimality already implies strong $o$\--minimality. They also
left open several questions, such as the existence of a cell
decomposition. 

Mourgues proved in \cite{mour-2009} that a cell decomposition similar
to the one of Denef in \cite{dene-1984} holds for a strongly
$p$\--minimal field $K$ if and only if it has {\df definable Skolem
functions} (``definable selection'' in \cite{mour-2009}), that is if
for every positive integers $m,n$ and every definable subset $S$ of
$K^{m+n}$ the coordinate projection of $S$ onto $K^m$ has a definable
section. It is not known at the moment whether strongly $p$\--minimal
fields always have definable Skolem functions.

As Cluckers noted in \cite{cluc-2004}, a preparation theorem for
definable functions was lacking in \cite{mour-1999}. This remark
applies as well to \cite{mour-2009}. Cluckers filled this lacuna for
the classical analytic structure on $K$ (see below), and derived from
his preparation theorem several important applications, for parametric
integrals and classification of subanalytic sets up to definable
bijection. The former gives the rationality of the Poincar\'e series of
a restricted analytic function. It has been generalised recently to
strongly $p$\--minimal fields in \cite{cubi-leen-2015}, by means of a
slightly different preparation theorem for definable functions.
However this preparation theorem and the cell decomposition that it
uses, are weaker than the original ones studied by Denef, Mourgues and
Cluckers. In particular they do not imply the existence of definable
Skolem functions, and neither the classification of definable sets up
to definable bijection. 

The aim of this paper is to address some of these questions by
introducing another notion of minimality for expansions of
$p$\--adically closed fields, called ``$p$\--optimality'' (see
definition below) with the following properties:
\begin{enumerate}
  \item
    It is {\em intrinsic} (that is its definition only involves the
    given structure, not those which are elementarily equivalent to
    it) {\em natural} and {\em general} enough to include all the
    known examples of $p$\--minimal fields.
  \item 
    Nevertheless it implies strong $p$\--minimality, the existence of
    definable Skolem functions, cell decomposition and (under a mild
    assumption which we will discuss in Remark~\ref{re:mild-min-val})
    cell preparation, so that all the applications of \cite{cluc-2004}
    generalize to this context.
\end{enumerate}

\paragraph{Acknowledgement.} 

We would like to thank Raf Cluckers and Pablo Cubides-Kovacsiks for
helpful discussions. This paper is based on \cite{hask-macp-1997} and
\cite{dene-1984}, with which the reader is expected to be familiar. We
will also make extensive use of \cite{cluc-2003}. Moreover we borrowed
ideas from papers of other authors, especially Raf Cluckers in
\cite{cluc-2004}. The concept of $p$\--optimal field seems to be new
but appears implicitly in many papers on $p$\--adic fields, especially
\cite{dene-1986} which has been a source of inspiration for us.

\paragraph{Defining $p$\--optimal fields.}

By a celebrated theorem of Macintyre \cite{maci-1976} (generalized to
$p$\--adically closed fields in \cite{pres-roqu-1984}) when $K=\QQ_p$ every
semi-algebraic subset of $K^m$ is a (finite) boolean combination of
sets of the form 
\begin{equation}\label{eq:def-basic}
  S=\big\{x\in K^m\tq f(x)\in P_N \big\}
\end{equation}
with $f$ a polynomial function, $N\geq1$ an integer and
\begin{displaymath} 
  P_N=\big\{x\in K\tq \exists y\in K,\ x=y^N\big\}.
\end{displaymath}
We define {\df $d$\--basic functions} as $m$\--ary functions for some
$m$ which are polynomial in the last $d$ variables with as
coefficients global definable functions in the $m-d$ first variables,
and {\df $d$\--basic sets (of power $N$)} as the sets of the same form
as (\ref{eq:def-basic}) with $d$\--basic functions instead of 
polynomial\footnote{Note that a global function in $m$ variables is
  $m$\--basic if and only if it is polynomial, hence Macintyre's
  theorem can be rephrased as: every semi-algebraic subset of $K^m$ is
  $m$\--basic.\label{fo:mac-rephrased}} functions.
When $d=1$ we simply talk about basic functions and sets. We say that
$(K,\cL)$ (or simply $K$ for short) is {\df $p$\--optimal} if every
definable subset of $K^m$ is a (finite) boolean combination of basic
sets, for every $m$.

\begin{remark}\label{re:bool-basic}
  By the argument of Lemma~2.1 in \cite{dene-1984}, the following subsets
  of $K^m$ are $d$\--basic, for every $d$\--basic $m$\--ary functions $f$,
  $g$.
  \begin{displaymath}
    \big\{x\in K^m\tq f(x)=0 \big\} \mbox{ \ and \ } 
    \big\{x\in K^m\tq |g(x)|\leq|f(x)| \big\} 
  \end{displaymath}
  Moreover, since $P_N^*=P_N\setminus\{0\}$ is a subgroup of finite index in
  $K^*$, the complement in $K^m$ of a $d$\--basic set is a finite
  union of $d$\--basic sets. Hence every (finite) boolean combination
  of basic sets is the union of intersections of finitely many basic
  sets. All of them can be taken of the same power, because
  $P_{N'}^*$ is a subgroup of $P_N^*$ of finite index for every $N'$
  which is divisible by $N$. 
\end{remark}

\paragraph{(Strong) $p$\--minimality versus $p$\--optimality.} Note
that $p$\--optimal fields are {\em not assumed to be} strongly
$p$\--minimal. They are $p$\--minimal because basic
subsets of the affine line $K$ are semi-algebraic. Moreover it is
difficult to imagine any proof of $p$\--minimality which does not
involve in a way or another a quantifier elimination result similar to
Macintyre's Theorem. The condition defining $p$\--optimality is
actually very close to such kind of elimination. So close that we can
expect it to be proved simultaneously in most cases, if not all,
without additional effort.
Although not surprising, it is then quite remarkable that every
$p$\--optimal field is strongly $p$\--minimal. More precisely,
recalling that $(K,\cL)$ is an expansion of a $p$\--adically closed
field we have (Theorem~\ref{th:equiv-p-opt}):

\begin{theorem}\label{th:intro-equiv-p-opt}
  The following are equivalent:
  \begin{enumerate}
    \item\label{it:intro-p-opt}
      $(K,\cL)$ is $p$\--optimal.
    \item\label{it:intro-CD}
      Denef's Cell Decomposition Theorem~\ref{th:cell-dec-PN} holds in
      $(K,\cL)$.
    \item\label{it:intro-strong}
      $(K,\cL)$ is strongly $p$\--minimal and has definable Skolem
      functions.
  \end{enumerate}
\end{theorem}

Of course (\ref{it:intro-strong})$\Rightarrow$(\ref{it:intro-CD}) follows from
\cite{mour-2009} (not the other implications, because Mourgues
considers only strongly $p$\--minimal fields). Since every known
example of $p$\--minimal field is strongly $p$\--minimal and has
definable Skolem functions, Theorem~\ref{th:intro-equiv-p-opt} shows
that all of them are indeed $p$\--optimal.

\paragraph{Main other results.} 
Remember that, identifying any subset of $K^m\times|K|^d$ with its inverse
image in $K^{m+d}$ by the valuation, we call the former definable,
semi-algebraic, $d$\--basic, or basic, if the latter is so. Similarly
a function from $X\subseteq K^m$ to $|K|^d$ is definable or semi-algebraic if
its graph is so, in this broader sense. In Section~\ref{se:rel-p-min}
we will consider strongly $p$\--minimal fields satisfying the
following condition.
\begin{itemize}
  \item[(*)]
    Every continuous definable function from a closed and bounded
    definable set $X\subseteq K$ to $|K|\setminus\{0\}$ attains a minimum value. 
\end{itemize}
We call it the {\df  Extreme Value Property}. Note that it is not at
all a restrictive assumption: if $(K,\cL)$ is elementarily equivalent
to $(K',\cL)$ for some $p$\--adic field $K'$ then the Extreme Value
Property trivially holds true in $K'$ (because its $p$\--valuation
ring is compact), and passes to $K$ by elementary equivalence. It is
proved in \cite{cluc-2001} (Theorem~6) that if $(K,\cL)$ is strongly
$p$\--minimal then the definable subsets of $|K^d|$ are
semi-algebraic. The following is a ``relative'' version of this result
(Theorem~\ref{th:rel-p-min} and Corollary~\ref{co:pd-opt}). 

\begin{theorem}\label{th:intro-pmd}
  If $(K,\cL)$ is strongly $p$\--minimal and satisfies the 
  Extreme Value Property, then every definable set $S\subseteq K\times|K|^d$ is
  semi-algebraic. If moreover $K$ is $p$\--optimal then every
  definable subset of $K^m\times|K|^d$ is a boolean combination of
  $(d+1)$\--basic sets.
\end{theorem}

In Section~\ref{se:cell-prep} we derive from it a preparation
Theorem~\ref{th:cell-prep} for definable functions, analogous to
Theorem~2.8 in \cite{cluc-2004}. As an application we get
(Theorem~\ref{th:isom-dim}):

\begin{theorem}\label{th:intro-isom}
  Two infinite sets definable over a $p$\--optimal field satisfying
  the  Extreme Value Property are isomorphic\footnote{Following
    \cite{cluc-2001} we call ``isomorphism'' the definable
    bijections.\label{fm:isom-bij}} if and only if they have
  the same dimension.
\end{theorem}

\begin{remark}\label{re:mild-min-val}
  As already mentioned the Extreme Value Property is not a strong
  assumption. In particular it holds true for every semi-algebraic
  function in a $p$\--adically closed field (by reduction to the
  $p$\--adic case, with the same argument as above). Moreover the Cell
  Preparation Theorem~\ref{th:cell-prep} applied to any unary
  definable function $f$ from a closed and bounded set $S\subseteq K$ to
  $K\setminus\{0\}$ gives that the function $|f|:S\to |K|\setminus\{0\}$ is semi-algebraic,
  hence has a minimum value. So the Cell Preparation Theorem holds
  true in a $p$\--optimal field if and only if it satisfies the
  Extreme Value Property. 
\end{remark}

\paragraph{Other terminology and notation.} 

For convenience we will sometimes add to $K$ one more element $\infty$,
with the property that $|x|<|\infty|$ for every $x$ in $K$. We also denote by
$\infty$ any partial function with constant value $\infty$.

Topological notions refer to the topology of the $p$\--valuation,
or its image in $|K|$. 

For every subset $X$ of $K$ we let $X^*=X\setminus\{0\}$. Note the
difference between $R^*=R\setminus\{0\}$ and $R^\times=$ the set of units in $R$.

Recall that $K^0$ is a one-point set. When a tuple $a=(x,t)$ is given
in $K^{m+1}$ it is understood that $x=(x_1,\dots,x_m)$ and $t$ is the last
coordinate. We let $\widehat{a}=x$ denote the projection of $a$ onto
$K^m$. Similarly, the projection of a subset $S$ of $K^{m+1}$ onto
$K^m$ is denoted by $\widehat{S}$. 

We extend $|.|$ (or $v$) to $K^m$ coordinatewise. That is, for every
$x\in K^m$ we let: 
\begin{displaymath}
  \big|(x_1,\dots,x_m)\big|=\big(|x_1|,\dots,|x_m|\big).
\end{displaymath}
For every $A\subseteq K^m$ we let 
$|A|$ denote the image of $A$ by this extension of the valuation.

For every integer $e\geq1$ let $\UU_e=\{x\in K\tq x^e=1\}$.
Analogously to Landau's notation $\cO(x^n)$ of calculus, we let
$\cU_{e,n}(x)$ denote {\em any} definable function in the
multi-variable $x$ with values in $(1+\pi^nR)\UU_e$. So, given a family
of functions $f_i$, $g_i$ on the same domain $X$, we write that
$f_i=\cU_{e,n}g_i$ for every $i$, when there are definable functions
$\omega_i:X\to R$ and $\chi_i:X\to\UU_e$ such that for every $x$ in $X$,
$f_i(x)=\big(1+\pi^n\omega_i(x)\big)\chi_i(x)g_i(x)$. When $e=1$,
$\cU_{1,n}(x)$ is simply written $\cU_n(x)$.

If $K^\circ$ is a finite extension of $\QQ_p$ to which $K$ is elementarily
equivalent as a ring, and $R^\circ$ is the $p$\--valuation ring of
$K^\circ$, then the following set is semi-algebraic (see Lemma~2.1,
point 4, in \cite{dene-1986})
\begin{displaymath}
  Q^\circ_{N,M}=\{0\} \cup \bigcup_{k\in\ZZ}\pi^{kN}(1+\pi^M R^\circ).
\end{displaymath}
We let $Q_{N,M}$ denote the semi-algebraic subset of $K$
corresponding\footnote{For a more intrinsic definition of $Q_{N,M}$
inside $K$, see~\cite{cluc-leen-2012}.} 
by elementary equivalence to $Q^\circ_{N,M}$ in $K$. If $M>2v(N)$, Hensel's
lemma implies that $1+\pi^MR$ is contained in $P_N^*$. Note that
in this case, $Q_{N,M}^*$ is a clopen subgroup of $P_N^*$ with finite index.
The next property also follows from Hensel's lemma (see for example
Lemma~1 and Corollary~1 in \cite{cluc-2001}).

\begin{lemma}\label{le:Hensel-DP}
  The function $x\mapsto x^e$ is a group endomorphism of $Q_{N_0,M_0}^*$. If
  $M_0>v(e)$ this endomorphism is injective and its image is
  $Q_{eN_0,v(e)+M_0}^*$.  
\end{lemma}

In particular $x\mapsto x^N$ defines a continuous bijection from
$Q_{1,v(N)+1}$ to $Q_{N,2v(N)+1}$. We let $x\mapsto x^\frac{1}{N}$ denote
the reverse bijection.

\section{Cell decomposition}
\label{se:cell-dec}

This section gives an overview of the techniques used in Denef's cell
decomposition. We emphasize that they do not only apply to polynomial
functions, as in \cite{dene-1984}, but also to basic functions. This
allows us to extend Denef's cell decomposition of semi-algebraic sets
over $p$\--adic fields to definable sets over $p$\--optimal fields
(Theorem~\ref{th:cell-dec-PN}).

The cells which usually appear in the literature on $p$\--adic fields
are non empty subsets of $K^{m+1}$ of the form: 
\begin{equation}\label{eq:def-cell}
  \{(x,t)\in X\times K\tq |\nu(x)| \Box_1 |t-c(x)| \Box_2 |\mu(x)| \mbox{ and } t-c(x)\in\lambda
  G\}
\end{equation}
where $X\subseteq K^m$ is a definable set, $c,\mu,\nu$ are definable functions
from $X$ to $K$, $\Box_1,\Box_2$ are $\leq,<$ or no condition, $\lambda\in K$ and $G$
is a semi-algebraic subgroup of $K^*$ with finite index. In this paper
we will only consider the cases when $G$ is $K^*$
(Theorem~\ref{th:cell-dec}), $P_N^*$ (Theorem~\ref{th:cell-dec-PN}) or
$Q_{N,M}^*$ (Theorem~\ref{th:cell-prep}).

In its simplest form, Denef's Cell Decomposition Theorem asserts that
every semi-algebraic subset of $K^m$ is the disjoint union of finitely
many cells. It will be convenient to fix a few more conditions
on our cells, but most of all we want to pay attention on {\em how the
functions defining the output cells depend on the input data}. 

So we define {\df presented cells} in $K^{m+1}$ as tuples
$A=(c_A,\nu_A,\mu_A,\lambda_A,G_A)$ with $c_A$ a definable function on a nonempty
domain $X\subseteq K^m$ with values in $K$, $\nu_A$ and $\mu_A$ either definable
functions on $X$ with values in $K^*$ or constant functions on $X$
with values $0$ or $\infty$, $\lambda_A$ an element of $K$ and $G_A$ 
semi-algebraic subgroup of $K^*$ with finite index, such that for every $x\in
X$ there is $t\in K$ such that:
\begin{equation}\label{eq:cell-cond}
  |\nu_A(x)|\leq|t-c_A(x)|\leq|\mu_A(x)|\mbox{ \ and \ } t-c_A(x)\in \lambda_A G_A.
\end{equation}
Of course the set of tuples $(x,t)\in X\times K$ satisfying
(\ref{eq:cell-cond}) is a cell of $K^{m+1}$ in the usual sense of
(\ref{eq:def-cell}). We call it the {\df underlying cellular set} of
$A$. Abusing the notation we will most often also denote that set by $A$. The
existence, for every $x\in X$, of $t$ satisfying (\ref{eq:cell-cond}) 
simply means that $X$ is exactly $\widehat{A}$. We call it the {\df
base} of $A$. The function $c_A$ is called its {\df center}, $\mu_A$
and $\nu_A$ its {\df boundaries}.
We also speak of a {\df presented cell mod $G$} when $G_A=G$.

A presented cell $A$ is said to be of {\df type} $0$ if $\lambda_A=0$, and
of type $1$ otherwise. Contrary to its center, boundaries, and modulo, the
type of $A$ only depends on its underlying set.

{\em The word ``cell'' will usually refer to presented cells}. However, for
sake of simplicity, we will freely talk of disjoint cells, bounded
cells, families of cells partitioning some set and so on, meaning that
the underlying cellular sets of these (presented) cells have the
corresponding properties. For instance, it is clear that every
cellular set as in (\ref{eq:def-cell}) is in that sense the disjoint
union of finitely many (presented) cells mod $G$. 

\begin{lemma}[Denef]\label{le:skolem1}
  Let $S$ be a definable subset of $K^{m+n}$. Assume that there is an
  integer $\alpha\geq1$ such that for every $x$ in $K^m$ the fiber
  \begin{displaymath}
    S_x=\big\{y\in K^n\tq (x,y)\in S\big\} 
  \end{displaymath}
  has cardinality $\leq\alpha$. Then the coordinate projection of $S$ on $K^m$
  has a definable section.
\end{lemma}

\begin{proof}
Identical to the proof of Lemma~7.1 in \cite{dene-1984}.
\end{proof}

\begin{lemma}[Denef]\label{le:denef-7-2}
  Let $f$ be an $(m+1)$\--ary basic function with variables
  $(x,t)=(x_1,\dots,x_m,t)$. Let $n\geq1$ be a fixed integer. Then there
  exists a finite partition of $K^{m+1}$ into sets $A$ of the form
  \begin{displaymath}
    A=\bigcap_{j\in S}\bigcap_{l\in S_j}\big\{(x,t)\in K^{m+1}\tq x\in C\mbox{ and }
    |t-c_j(x)|\square_{j,l}|a_{j,l}(x)|\big\}
  \end{displaymath}
  where $S$ and $S_j$ are finite index sets, $C$ is a definable subset
  of $K^m$, and $c_j$, $a_{j,l}$ are definable functions from $K^m$ to
  $K$, such that for all $(x,t)$ in $A$ we have
  \begin{displaymath}
    f(x,t)=\cU_n(x,t)h(x)\prod_{j\in S}\big(t-c_j(x)\big)^{e_j}
  \end{displaymath}
  with $h:K^m\to K$ a definable function and $e_j\in\NN$. 
\end{lemma}

It is sufficient to check it for every $n$ large enough so we can
assume that:
\begin{equation}\label{eq:D1-n-grand}
  1+\pi^nR\subseteq P_N\cap R^\times
\end{equation}
Thus $\cU_n(x,t)$ in the conclusion could be replaced by $u(x,t)^N$
with $u$ a definable function from $A$ to $R^\times$. This is indeed how
this result is stated in Lemma~7.2 of \cite{dene-1984}. However it
is the above equivalent (but slightly more precise) form which appears
in Denef's proof, and which we retain in this paper. 

\begin{proof}
The proof is exactly the same as the one of Lemma~7.2 of
\cite{dene-1984}. Of course, Lemma~7.1 used in Denef's proof has to be
replaced with the analogous Lemma~\ref{le:skolem1}. (Denef's result
assumes that $f$ is a polynomial, but the proof only uses that it's a
polynomial in the last variable, so it also applies to basic $f$.)
\end{proof}

\begin{remark}\label{re:denef-coalg}
  \textbf{(co-algebraic functions)} 
  A remarkable by-product of Denef's proof is that the functions $c_j$
  and $a_{j,l}$ in the conclusion of Lemma~\ref{le:denef-7-2} belong
  to $\coalg(f)$, which we define now. 
\end{remark}

Given a basic function $f$, we say that a function
$h:X\subseteq K^m\to K$ belongs to $\coalg(f)$ if there exists a finite
partition of $X$ into definable pieces $H$, on each of which the degree
in $t$ of $f(x,t)$ is constant, say $e_H$, and such that the following
holds. If $e_H\leq 0$ then $h(x)$ is identically equal to $0$ on $H$.
Otherwise there is  a family $(\xi_1,\dots,\xi_{r_H})$ of $K$\--linearly
independent elements in an algebraic closure of $K$ and a family of
definable functions $b_{i,j}:H\to K$ for $1\leq i\leq e_H$ and $1\leq j\leq r_H$,
and $a_{e_H}:H\to K^*$ such that for every $x$ in $H$
\begin{displaymath}
  f(x,T)=a_{e_H}(x)\prod_{1\leq i\leq e_H}
         \bigg(T-\sum_{1\leq j\leq r_H}b_{i,j}(x)\xi_j\bigg)
\end{displaymath}
and
\begin{displaymath}
  h(x)=\sum_{1\leq i\leq e_H}\sum_{1\leq j\leq r_H}\alpha_{i,j}b_{i,j}(x)
\end{displaymath}
with the $\alpha_{i,j}$'s in $K$. If $\cF$ is any family of
basic functions we let $\coalg(\cF)$ denote the set of
linear combinations of functions in $\coalg(f)$ for $f$ in $\cF$.

\begin{theorem}[Denef]\label{th:cell-dec} 
  Let $\cF$ be a finite family of $(m+1)$\--ary basic functions.
  Let $n\geq1$ be a fixed integer. Then there exists a finite partition
  of $K^{m+1}$ into presented cells $H$ mod $K^*$ such that the center
  and boundaries of $H$ belong to $\coalg(\cF)\cup\{\infty\}$ and for every $(x,t)$
  in $H$ and every $f$ in $\cF$
  \begin{equation}\label{eq:cell-dec}
    f(x,t)=\cU_n(x,t)h_{f,H}(x)\big(t-c_H(x)\big)^{\alpha_{f,H}}
  \end{equation}
  with $h_{f,H}:\widehat{H}\to K$ a definable function and $\alpha_{f,H}\in\NN$.
\end{theorem}

\begin{proof}
Follow the proof of Theorem~7.3 in \cite{dene-1984}, using once again
basic functions instead of polynomial functions.
\end{proof}

Given two families $\cA$, $\cB$ of subsets of $K^m$, recall that $\cB$
{\df refines} $\cA$ if $\cB$ is a partition of $\bigcup\cA$ such that every $A$ in
$\cA$ which meets some $B$ in $\cB$ contains it. 

\begin{corollary}[Denef]\label{co:cell-dec-PN}
  Let $\cF$  be a finite family of $m$\--ary basic functions, 
  $N\geq1$ an integer and $\cA$ a family of boolean combinations of
  subsets of $K^m$ defined by $f(x)\in P_N$ with $f$ in $\cF$. Then
  there exists a finite family $\cH$ of cells mod $P_N^*$ with center
  and boundaries in $\coalg(\cF)$ which refines $\cA$. 
\end{corollary}

\begin{proof}
Theorem~\ref{th:cell-dec} applies to $\cF$ with $n>2v(N)$, so that
$1+\pi^nR\subseteq P_N$. It gives a partition of $K^m$ into presented cells $B$
mod $K^*$. Every such cell $B$ is the disjoint union of finitely many
presented cells $H$ mod $P_N^*$, whose centers and boundaries are the
restrictions to $\widehat{H}$ of the center and boundaries of $B$ (hence
belong to $\coalg(\cF)$), on which $h_{f,B}(x)P_N^*$ and
$(t-c_B(x))P_N^*$ are constant, simultaneously for every $f$ in $\cF$.
Thus every $A$ in $\cA$ either contains $H$ or is disjoint from $H$ by
(\ref{eq:cell-dec}) and our choice of $n$, which proves the result. 
\end{proof}

The following simpler statement, which follows directly from
Corollary~\ref{co:cell-dec-PN} by $p$\--optimality, is sufficient in
most cases.

\begin{theorem}[Denef's cell decomposition]\label{th:cell-dec-PN} 
  If $(K,\cL)$ is $p$\--optimal, then for every finite
  family $\cA$ of definable subsets of $K^m$ there is for some $N$ a
  finite family of presented cells mod $P_N^*$ refining $\cA$.
\end{theorem}

\begin{remark}\label{re:piece-cont}
  It has been proved in \cite{darn-coba-leen-2015-tmp} that every
  definable function in a strongly $p$\--minimal field is piecewise
  continuous. We will show in the next section that $p$\--optimal
  fields are strongly $p$\--optimal. Thus the boundaries and centers
  of the cells in the above cell decompositions can be chosen
  continuous by refining appropriately a given cell decomposition. 
\end{remark}

\section{From $p$\--optimality to strong $p$\--minimality with Skolem functions}
\label{se:skolem}

\begin{lemma}\label{le:skolem}
  Assume that Denef's Cell Decomposition Theorem~\ref{th:cell-dec-PN}
  holds true for $(K,\cL)$. Then it has definable Skolem functions.
\end{lemma}

The proof is taken from the appendix of \cite{dene-drie-1988}. 
It is similar to proposition~4.1 in \cite{mour-2009} except that we do
not assume strong $p$\--minimality (nor any continuity in the boundaries
of the cells). 

\begin{proof}
By a straightforward induction it suffices to prove that for every
definable subset $A$ of $K^{m+1}$ the coordinate projection of $A$
onto $\widehat{A}$ has a definable section. If $A$ is a union of
finitely many definable sets $B$ and if a definable section
$\sigma_B:\widehat{B}\to B$ has been found for each projection of $B$ onto
$\widehat{B}$ we are done. Thus, by cell decomposition, we can assume
that $A$ is a presented cell mod $P_N^*$ for some $N$. We deal with
the case when $A=(c_A,\nu_A,\mu_A,\lambda_A)$ is of type $1$ and $\nu_A\neq0$ or
$\mu_A\neq\infty$, the other cases being trivial. 

If $\nu_A\neq0$, as $P_N^*$ is a definable subgroup of $K^\times$ with finite
index, there is a partition of $\widehat{A}$ into finitely many
definable pieces $X$ on each of which $\nu_A/\lambda_A$ has constant residue
class modulo $P_N^*$. Again it suffices to prove the result for each
piece $A\cap(X\times K)$ of $A$. So we can assume that $X=\widehat{A}$, that
is $\nu_A(x)/\lambda_A\in aP_N^*$ for some constant $a\in K^\times$ and every $x\in
\widehat{A}$. Moreover we can choose $a$ so that $v(a)$ is a
non-negative integer $k<N$. Let $\tau:x\in\widehat{A}\to c_A(x)+\nu_A(x)/a$. If
$(x,\tau(x))\in A$ for every $x\in\widehat{A}$ we are done, since
$\sigma:x\in\widehat{A}\mapsto(x,\tau(x))$ is then a definable section of the
coordinate projection of $A$ onto $\widehat{A}$. So let us prove this.
  
Since $\tau(x)-c_A(x)=\lambda_A(\nu_A(x)/(a\lambda_A))$, it belongs to $\lambda_A P_N^\times$ by
construction. Obviously we also have $|\nu_A(x)|\leq|\nu_A(x)/a|$ because $a\in
R$, and thus $|\nu_A(x)|\leq|\tau(x)-c_A(x)|$. It remains to check that
$|\tau(x)-c_A(x)|\leq|\mu_A(x)|$, that is $|\nu_A(x)/a|\leq|\mu_A(x)|$. Pick any $t\in
K^\times$ such that $(x,t)\in A$. We have $|\nu_A(x)|\leq|t-c_A(x)|\leq|\mu_A(x)|$, so
it suffices to check that $|\nu_A(x)/a|\leq|t-c_A(x)|$, that is
$v(\nu_A(x))-k\geq v(t-c_A(x))$. Let $\delta=(t-c_A(x))/\lambda_A$, since $(x,t)\in A$
we have $v(\nu_A(x)/\lambda_A)\geq v(\delta)$ and $v(\delta)\in v(P_N^*)=N\cZ$. By
construction we also have $v(\nu_A(x)/\lambda_A)\in v(aP_N^*)=k+N\cZ$.
Altogether, since $0\leq k<N$, this implies that $v(\nu_A(x)/\lambda_A)\geq v(\delta)+k$.
So $v(\nu_A(x))-k\geq v(\delta)+v(\lambda_A) = v(t-c_A(x))$, which finishes the proof
in this case.

If $\nu_A=0$ and $\mu_A\neq\infty$ a similar argument on $\mu_A$ gives the
conclusion. 
\end{proof}

\begin{theorem}\label{th:equiv-p-opt}
  The following are equivalent:
  \begin{enumerate}
    \item\label{it:p-opt}
      $(K,\cL)$ is $p$\--optimal.
    \item\label{it:CD}
      Denef's cell decomposition Theorem~\ref{th:cell-dec-PN} holds in
      $(K,\cL)$.
    \item\label{it:sko}
      $(K,\cL)$ is strongly $p$\--minimal and has definable Skolem function.
  \end{enumerate}
\end{theorem}

\begin{proof}
(\ref{it:p-opt})$\Rightarrow$(\ref{it:CD}) is Theorem~\ref{th:cell-dec-PN}. 
Let us prove that (\ref{it:CD})$\Rightarrow$(\ref{it:sko}). 
By Lemma~\ref{le:skolem} it only
remains to derive strong $p$\--minimality from the Cell Decomposition
Theorem~\ref{th:cell-dec-PN}. 

Let $\Phi(\xi,\sigma)$ be a parameter-free formula with $m+1$ variables. It
defines a subset $S$ of $K^{m+1}$ which splits into finitely many
cells $C$ mod $P_N^*$ for some $N$. Let $\cC$ be the family of these
cells, and $X_1,\dots,X_r$ a finite partition of $\widehat{S}$ refining
the $\widehat{C}$'s for $C\in\cC$. For each $i\leq r$ let $\theta_i(\alpha_i,\xi)$ be a
parameter-free formula in $n_i+m$ variables and $a_i\in K^{n_i}$ such
that 
\begin{displaymath}
  X_i=\{x\in K^m\tq K\models\theta_i(a_i,x)\}. 
\end{displaymath}
Let $\Theta(\alpha_1,\dots,\alpha_r)$ be the parameter-free formula in $n_1+\cdots+n_r$ 
variables saying that, given any values $a'_i$ of the parameters
$\alpha_i$, the formulas $\theta_i(a'_i,\xi)$ define a partition of $\widehat{S}$.
In particular we have $K\models\Theta(a_1,\dots,a_r)$. 

Let $\cC_i$ be the family of all the cells $C\cap(X_i\times K)$ for $C\in\cC$.
This is a finite partition of $S\cap(X_i\times K)$ into cells mod $P_N^*$,
which consists in $k_0^i$ cells of type $0$, $k_1^i$ cells $D$ of type $1$
with $\mu_D\neq\infty$, and $k_\infty^i$ cells $D$ of type $1$ with $\mu_D=\infty$. We let
$k^i=(k_0^i,k_1^i,k_\infty^i)$. For every $x\in X_i$, the fiber $S_x=\{t\in
K\tq(x,s)\in S\}$ is the disjoint union of the fibers $C_x$ for
$C\in\cC_i$, each of which is of the same type as $C$. Given a tuple
$k=(k_0,k_1,k_\infty)$ it is an easy exercise to write a parameter-free
formula $\Psi_{k,N}(\xi)$ in $m$ free variables saying that, given any
value $x'$ of the parameter $\xi$, the set of points $t'$ in $K$ such
that $K\models\Phi(x',t')$ is the disjoint union of $k_0$ cells mod $P_N^*$ of
type $0$, $k_1$ cells $D'$ mod $P_N^*$ of type $1$ with $\mu_{D'}\neq\infty$, and
$k_\infty$ cells $D'$ mod $P_N^*$ of type $1$ with $\mu_{D'}=\infty$. By construction
we have
\begin{displaymath}
  K \models \exists \alpha_1,\dots,\alpha_r\, \Theta(\alpha_1,\dots,\alpha_r) \land \lland_{i\leq r} \forall \xi\,
  \big[\theta_i(\alpha_i,\xi)\to\Psi_{k^i,N}(\xi)\big]
\end{displaymath}

This formula is satisfied in every $\tilde K\equiv K$. So there are $\tilde
a_i$ in $\tilde K^{n_i}$ for $i\leq r$ such that the sets 
\begin{displaymath}
  \tilde X_i=\{\tilde x\in \tilde K^m\tq \tilde K\models\theta_i(\tilde a_i,\tilde x)\}
\end{displaymath}
form a partition of $\{\tilde x\in\tilde K^m\tq \exists\tilde t\in \tilde K,\ 
\tilde K\models\Phi(\tilde x,\tilde t)\}$, and for every $\tilde x\in\tilde X_i$
the set of $\tilde t\in \tilde K$ such that $\tilde K\models\theta_i(\tilde x,\tilde
t)$ is the disjoint union of $k_0^i+k_1^i+k_\infty^i$ cells of $\tilde K$.
In particular the formula $\Phi(\tilde x,\tau)$ defines a semi-algebraic
subset of $\widetilde{K}$, whatever is the value of the parameter $\tilde x$ in
$\tilde K^m$. This being true for every formula $\Phi$, it follows that
$\tilde K$ is $p$\--minimal hence that $K$ is strongly $p$\--minimal. 

Finally let us prove that (\ref{it:sko})$\Rightarrow$(\ref{it:p-opt}).
Let $S$ be a definable subset of $K^{m+1}$, and $S'$ the corresponding
definable set in an elementary extension $K'$ of $K$. For every $x'$
in $K'^m$ let $S'_{x'}$ denote the fiber of $\widehat{S'}$ over $x'$:
\begin{displaymath}
  S'_{x'}=\big\{t'\in K'\tq (x',t')\in S'\big\} 
\end{displaymath}
For every $x'$ in $\widehat{S'}$ the $p$\--minimality of $K'$ 
and Macintyre's theorem (see Footnote~\ref{fo:mac-rephrased}) give a
tuple $z'_{x'}$ of coefficients of a description of $S'_{x'}$ as a
boolean combination of basic sets. The model-theoretic Compactness
Theorem then gives definable subsets $X_1,\dots,X_q$ partitioning
$K^m$ and for
every $i\leq q$ an $\cL$--formula $\varphi_i(x,t,z)$ with $m+1+n_i$ free
variables which is a boolean combination of formulas of the form $f(x,t,z)\in P_N$
with $f\in\ZZ[x,t,z]$, such that for every $x$ in $X_i$ there is a list of
coefficients $z_x$ such that
\begin{displaymath}
  S_x=\big\{ t\in K\tq K\models\varphi(x,t,z_x) \big\}.
\end{displaymath}
In other words, for every $x$ in $X_i$
\begin{displaymath}
  K\models\exists z\;\forall t\ \big((x,t)\in S \leftrightarrow \varphi_i(x,t,z)\big).
\end{displaymath}
Our assumption~(\ref{it:sko}) then gives for each $i\leq q$
a definable function $\zeta_i:X_i\to K^{n_i}$ such that for every $x\in X_i$
\begin{displaymath}
  K\models\forall t\ \big[(x,t)\in S \leftrightarrow \varphi_i\big(x,t,\zeta_i(x)\big)\big].
\end{displaymath}
Let $B_i=\{(x,t)\in K^{m+1}\tq K\models\varphi_i(x,t,\zeta_i(x))\}$. By construction this
is a boolean combination of basic subsets of $K^{m+1}$, hence so is
$C_i=B_i\cap(X_i\times K)$. The conclusion follows, since $S$ is the union of
these $C_i$'s.
\end{proof}

\section{Relative $p$\--minimality}
\label{se:rel-p-min}

The aim of this section is to prove the following result. It may be
called ``relative $p$\--minimality''. 

\begin{theorem}\label{th:rel-p-min}
  Assume that $(K,\cL)$ is strongly $p$\--minimal and satisfies the Extreme
  Value Property. Then every definable set $S\subseteq K\times |K|^d$ is
  semi-algebraic, for every $d$. 
\end{theorem}

We need to state a few preliminary results and to introduce some notation. 
For every $a\in K$ and $r\in |K^*|$ we let 
\begin{displaymath}
  B(a,r)=\big\{y\in K\tq |x-y|<r\big\}
\end{displaymath}
denote the ball of center $a$ and radius $r$.

\begin{fact}\label{fa:sko-Z-gp}
  For every definable set $S\subseteq K^m\times|K|^d$, if $A\subseteq
  K^m$ is the image of the coordinate projection of $S$ onto $K^m$,
  there is a definable function $\sigma:A\to|K|^d$ such that $(x,\sigma(x))\in
  S$ for every $x\in A$.
\end{fact}

\begin{proof}
By $p$\--minimality, the value group $v(K^*)$ is simply a $\ZZ$\--group.
Every nonempty definable subset of a $\ZZ$\--group which is bounded
above (resp.\ below) has a largest (resp.\ smallest) element. The
conclusion easily follows if $d=1$, and for $d\geq1$,
it is a straightforward
induction. 
\end{proof}

Beware that $\sigma$ in Fact~\ref{fa:sko-Z-gp} {\em is not a Skolem
function} over $K$, because its codomain is in $|K|$.
The next Lemma shows that this can be fixed, in a strong sense. 

\begin{lemma}\label{le:SI}
  Assume that $(K,\cL)$ is strongly $p$\--minimal and satisfies the Extreme
  Value Property. Then every definable function $f:X\subseteq K\to |K|^d$ is
  semi-algebraic. In particular there is a semi-algebraic function
  $\tilde f:X\to K^d$ such that $f=|\tilde f|$.
\end{lemma}

For every $r\in|K^*|$ we let $r^+$ denote the element of $|K^*|$
immediately greater than $r$. 

\begin{proof}
If $f=(f_1,\dots,f_d)$ it suffices to prove the result separately for each
$f_i$, hence we can assume that $d=1$. Given a finite partition of $X$
in definable pieces $Y$ it suffices to prove the result for the
restriction of $f$ to each $Y$ separately. Thus by splitting $X$ in
$f^{-1}(\{0\})$ and $X\setminus f^{-1}(\{0\})$ we can assume that $f(X)\subseteq|K^*|$. By
Theorem~3.3 and Remark~3.4 in \cite{hask-macp-1997} there is a
definable open set $U$ contained in $X$ such that $X\setminus U$ is finite and
$f$ is continuous on $U$. By throwing away a finite set if necessary,
we can therefore assume that $f$ is continuous and $X$ is open in $K$.
Finally we can assume that $f$ is not constant on $X$, otherwise the
result is trivial.

For every $a\in X$ the set of $r\in|K^*|$ such that $B(a,r)\subseteq X$ and $f$ is
constant on this ball is definable, nonempty and bounded above
(otherwise $X=K$ and $f$ is constant, which we have excluded) hence by
Fact~\ref{fa:sko-Z-gp} it has a maximum element $\rho(a)$. We are
claiming that the following set
\begin{displaymath}
  S=\big\{a\in X\tq\forall b\in B\big(a,\rho(a)^+\big)\cap X,\ f(a)\leq f(b)\big\}
\end{displaymath}
has the property that for every ball $B\subseteq X$ on which $f$ is
nonconstant, $B$ intersects both $S$ and $X\setminus S$. Indeed let
$B=B(c,r)$ be any such ball. The function $\rho$ is definable, so the
Extreme Value Property gives $a_0\in B$ such that $\rho(a_0)=\min_{b\in
B}\rho(b)$. Since $f$ is nonconstant on $B$, necessarily $\rho(a_0) <r$
hence $B(b,\rho(a_0)^+)\subseteq B$ for every $b \in B$. By construction $f$ is
nonconstant on $B(a_0,\rho(a_0)^+)$. The latter is the disjoint union
of $B(a_0,\rho(a_0))$ and finitely many balls $B(a_i,\rho(a_0))$ for $1\leq i\leq
n$ (where $n+1\geq2$ is the cardinality of the residue field). By
minimality of $\rho(a_0)$, $f$ is constant on each $B(a_i,\rho(a_0))$ hence
there is $i\neq j$ between $0$ and $n$ such that 
\begin{equation}\label{eq:S}
  \forall b\in B(a_0,\rho(a_0)^+),\ f(a_i) \leq f(b)\leq f(a_j).
\end{equation}
Moreover $f$ is nonconstant on the union of $B(a_k,\rho(a_0))$ for $0\leq k\leq n$
hence $f(a_i)< f(a_j)$. It follows that $\rho(a_i)=\rho(a_j)=\rho(a_0)$ and hence
$a_i\in S$ and $a_j\notin S$ by (\ref{eq:S}), which proves our claim. 

$X$ and $S$ are definable subsets of $K$, hence semi-algebraic by
$p$\--minimality. Thus there exists a partition $\cA$ of $X$ in
finitely many cells mod $Q_{N,M}^*$ for some $N,M$ such that $S$ is
also the union of the cells in $\cA$ that it contains. Every cell
$A\in\cA$ can be presented as the set of elements $t\in K$ such that 
\begin{displaymath}
  |\nu_A|\leq |t-c_A|\leq |\mu_A|\mbox{ and }t-c_A\in \lambda_AQ_{N,M}^*.
\end{displaymath}
We are claiming that $f(t)$ only depends on $|t-c_A|$ as $t$ ranges
over $A$. If $\lambda_A=0$ then $A$ is reduced to a point, hence $f$ is
constant on $A$. Otherwise $\lambda_A\neq0$ and for every $a\in K^\times$, we
  have to prove that $f$ is constant on the set $B_a$ of $t\in A$ such
  that $|t-c_A|=|a|$. We can assume that $B_a$ is nonempty,
  hence $|a|=|t_a-c_A|$ for some $t_a\in A$. Then $|\nu_A|\leq|a|\leq|\mu_A|$,
  hence $t\in B_a$ if and only if $|t-c_A|=|a|$ and $t-c_A\in\lambda_A
  Q_{N,M}^*$, that is $B_a=aR^\times\cap\lambda_AQ_{N,M}^*$. Pick any $b\in B_a$, then
  $bR^\times=aR^\times$ and $bQ_{N,M}^*=\lambda_AQ_{N,M}^*$ hence 
  \begin{displaymath}
    B_a= aR^\times\cap aQ_{N,M}^*=a(R^\times\cap Q_{N,M}^*)=a(1+\pi^MR). 
  \end{displaymath}
  In particular $B_a$ is a ball. So by construction of $\cA$, $A$
is either contained in $S$ or in $X\setminus S$ hence so is $B$. But then, by
construction of $S$, $f$ is constant on $B$. This proves our claim. 

Now pick any $A\in \cA$ and translate it by $c_A$. The result is a cell
$A'$ mod $Q_{N,M}^*$ centered at $0$ on which $f(t)$ only depends on
$|t|$. Thus the graph of the restriction $f_{|A}$ of $f$ to $A$ is the
intersection with $\lambda_AQ_{N,M}^*$ of the pre-image by the valuation of
a definable function $\theta:|A'|\to|K|$. By Theorem~6 in \cite{cluc-2003} it
follows that $f_{|A}$ is semi-algebraic, hence so is $f$. The last
point immediately follows from the existence of definable Skolem
functions for semi-algebraic sets (see for example \cite{drie-1984}).
\end{proof}

As already mentioned in the introduction, Theorem~\ref{th:rel-p-min} is a
``relative'' version of Theorem~6 in \cite{cluc-2003}. Since our proof
heavily depends on the main results of \cite{cluc-2003} it is more
convenient here to use additive notation for the value group, so let
$G=v(K^*)$. Theorem~6 in \cite{cluc-2003} actually says that for every
definable set $S\subseteq (K^*)^d$, with $(K,\cL)$ a strongly $p$\--minimal expansion
of a $p$\--adically closed field, the image of $S$ in $G^d$ by the
valuation is definable in Presburger language
\begin{displaymath}
  \cL_{Pres}=\{0,1,+,\leq,(\equiv_n)_{n>0}\} 
\end{displaymath}
where $\equiv_n$ is interpreted in $G$ as the binary congruence relation modulo
the integer $n$. 

It follows from Theorem~1 in \cite{cluc-2003} and Remarks~(iii) just
above it that every subset of $G^d$ definable in the language
$\cL_{Pres}$ is the union of finitely many disjoint sets defined by
the conjunction for $1\leq i\leq d$ of conditions $(E_i)$ of the form
\begin{displaymath}
    \zeta_i+\sum_{1\leq j< i}a_{i,j}\frac{X_j-c_j}{n_j} 
  \mathrel{\square_{i,1}} X_i \mathrel{\square_{i,2}}
  \zeta'_i+\sum_{1\leq j< i}a'_{i,j}\frac{X_j-c_j}{n_j} 
  \mbox{ and }X_i\equiv c_i\;[n_i]
\end{displaymath}
with every $\zeta_i,\zeta'_i\in G$, $a_{i,j},a'_{i,j},c_i,n_i\in\ZZ$, $0\leq c_i< n_i$
and $\square_{i,1},\square_{i,2}$ being either $\leq$ or no condition. Let $\lambda$ be the
list of all these integers and symbols. Let $\Lambda_d$ denote the set of
lists $\lambda$ of this sort. The conjunction of the above conditions
$(E_i)$ for $1\leq i\leq d$ is expressed by a formula $\varphi_\lambda(X,\zeta)$ with free
variables $X=(X_1,\dots,X_d)$ and parameters $\zeta=(\zeta_1,\dots,\zeta_d,\zeta'_1,\dots,\zeta'_d)$.
We let $\varphi_\lambda(X,Z)$ be the corresponding parameter-free formula in
$\cL_{Pres}$ with $d+2d$ free variables. 

With these results in mind we can turn to the proof of
Theorem~\ref{th:rel-p-min}. 

\begin{proof}
Let $S$ be a definable\footnote{Recall that in this context,
``definable'' means that the inverse image of $S$ by the valuation is
definable in $K\times K^d$.} subset of $K\times G^d$.
For every $x\in K$ the fiber $S_x=\{\tau\in G^d\tq(x,\tau)\in S\}$ is definable in
$\cL_{Pres}$ by Theorem~6 in \cite{cluc-2003}. 
Hence there is a finite set of elements $\lambda_1,\dots,\lambda_r\in\Lambda_d$ and 
parameters $\gamma_k\in G^{2d}$ such that the sets $C_{\lambda_k}(\gamma_k)$,
defined as the set of elements $\tau\in G^d$ such that $G\models
\varphi_{\lambda_k}(\tau,\gamma_k)$, form a partition of $S_x$. These formulas
$\varphi_{\lambda_k}(T,Z)$ easily translate into formulas $\psi_{\lambda_k}(T,Z)$
in the language of rings such that for every $t\in K^d$ and every $z\in
K^{2d}$, $K\models\psi_{\lambda_k}(t,z)$ if and only $G\models\varphi_{\lambda_k}(v(t),v(z))$. 

By strong $p$\--minimality the same holds true in every $(K',\cL)\equiv(K,\cL)$.
Hence by the model-theoretic Compactness Theorem there is a partition
of $K$ in finitely many definable sets $A_1,\dots,A_s$ and for each $l\leq s$
a finite set of indexes $\lambda_{1,l},\dots,\lambda_{r_l,l}\in\Lambda_d$ such that for every
$x\in A_l$ there are parameters $\zeta_{x,k,l}\in G^{2d}$ such that $S_x$
is partitioned by the sets $C_{\lambda_{k,l}}(\zeta_{x,k,l})$ for
$k\leq r_l$. By Fact~\ref{fa:sko-Z-gp} there are definable functions $\zeta_{k,l}$ 
from $A_l$ to $G^{2d}$ such that for every $x\in A_l$ the sets
$C_{\lambda_{k,l}}(\zeta_{k,l}(x))$ for $k\leq r_l$ form a partition of
$S_x$. By Lemma~\ref{le:SI} and the Extreme Value Property there are semi-algebraic functions
$\tilde z_{k,l}$ from $A_k$ to $K^{2d}$ such that
$\zeta_{k,l}=|\tilde z_{k,l}|$ (that is $\zeta_{k,l}=v\circ\tilde z_{k,l}$ with
additive notation). 

By the above construction $v^{-1}(S)$ is the disjoint union for $l\leq s$
and $k\leq r_l$ of the sets $B_{k,l}$ of tuples $(x,t)\in A_k\times K^d$ such
that $K\models \psi_{\lambda_{k,l}}(t,\tilde z_{k,l}(x))$. These sets are
semi-algebraic because $\psi_{\lambda_{k,l}}(T,Z)$ is a formula in the language
of rings and $\tilde z_{k,l}$ a semi-algebraic function. Thus
$v^{-1}(S)$ itself is semi-algebraic, hence so is $S$ by definition.
\end{proof}

\begin{corollary}\label{co:pd-opt}
  Assume that $K$ is $p$\--optimal and satisfies the Extreme Value
  Property. Then every definable subset of $K^m\times |K|^d$ is a boolean
  combination of $(d+1)$\--basic sets. 
\end{corollary}

\begin{proof}
If $m=1$ the conclusion follows from Theorem~\ref{th:rel-p-min} and
Macintyre's Theorem (see Footnote~\ref{fo:mac-rephrased}). 
Assume that it has been proved for $m\geq1$ and let $S$ be a definable
subset of $K^{m+1+d}$ which is the pre-image by the valuation of a
subset of $K^{m+1}\times|K|^d$. Let $S'$ be the
corresponding definable set over an elementary extension $K'$ of $K$.
For every $x'$ in $K'^m$ let $S'_{x'}$ denote the fiber of $S'$
over $x'$:
\begin{displaymath}
  S'_{x'}=\big\{(t',z')\in K'\times {K'}^d\tq (x',t',z')\in S'\big\} 
\end{displaymath}
This set $S'_{x'}$ is obviously the inverse image
in $K'\times K'^d$ by the valuation of a subset of $K'\times|K'|^d$. Note that
$K'$ is strongly $p$\--minimal and satisfies the Extreme Value Property,
because these two properties are preserved by elementary equivalence.
Thus Theorem~\ref{th:rel-p-min} applies in $K'$ and gives a tuple
$a'_{x'}$ of coefficients of a description of $S'_{x'}$ as a boolean
combination of $(d+1)$\--basic subsets of $K'^{d+1}$. The
model-theoretic Compactness Theorem then gives definable subsets
$A_1,\dots,A_q$ partitioning $K^m$, and for every $i\leq q$ an $\cL$--formula
$\varphi_i(\alpha,\tau,\zeta)$ with $n_i+1+d$ free variables which is a boolean
combination of formulas of the form $f(\alpha,\tau,\zeta)\in P_N$ with $f\in\ZZ[\alpha,\tau,\zeta]$,
such that for every $x$ in $A_i$ there is a list of coefficients $a_x$ such
that
\begin{displaymath}
  S_x=\big\{ (t,z)\in K\times K^d\tq K\models\varphi(a_x,t,z) \big\}.
\end{displaymath}
In other words, for every $x$ in $A_i$
\begin{displaymath}
  K\models\exists a\;\forall t,z\ \big((x,t,z)\in S \leftrightarrow \varphi_i(a,t,z)\big).
\end{displaymath}
By Theorem~\ref{th:equiv-p-opt}, $K$ has definable Skolem functions, hence for
each $i\leq q$ there is a definable function $\sigma_i:A_i\to K^{n_i}$ such that
for every $x\in A_i$
\begin{displaymath}
  K\models\forall t,z\ \big[(x,t,z)\in S \leftrightarrow \varphi_i\big(\sigma_i(x),t,z\big)\big].
\end{displaymath}
Let $B_i=\{(x,t,z)\in K^{m+1+d}\tq K\models\varphi_i(\sigma_i(x),t,z)\}$. By construction,
this is a boolean combination of $(d+1)$\--basic subsets of
$K^{m+1+d}$. On the other hand, $A_i\times K^{d+1}$ is obviously
a $(d+1)$\--basic subset of $K^{m+1+d}$. Indeed, if $c_i(x)$ denotes
the indicator function of $A_i$, then $h_i(x,t,z)=c_i(x)-1$ is
$(d+1)$\--basic and we have
\begin{displaymath}
  A_i\times K^{d+1} =\big\{(x,t,z)\in K^{m+1+d}\tq  h_i(x,t,z)=0\big\}
\end{displaymath}
which is a $(d+1)$\--basic set by Remark~\ref{re:bool-basic}. The
conclusion follows, since $S$ is the union of the sets  $B_i\cap(A_i\times
K^{d+1})$.
\end{proof}

\section{Cell preparation}
\label{se:cell-prep}

The main result of this section is the Cell Preparation
Theorem~\ref{th:cell-prep} for definable functions. We derive from it
our last main result, Theorem~\ref{th:isom-dim}, which classifies up to
definable bijections the definable sets over any $p$\--optimal field 
satisfying the Extreme Value Property.

\begin{lemma}[Denef]\label{le:norm}
  Assume that $K$ is $p$\--optimal and satisfies the Extreme Value
  Property. Then for every definable function $f:X\subseteq K^m\to K$ there is
  an integer $e\geq 1$ and a partition $\cA$ of $X$ in definable sets $A$
  such that for every $x$ in $A$
  \begin{displaymath}
    \big|f(x)\big|^e = \left|\frac{p_A(x)}{q_A(x)}\right|
  \end{displaymath}
  with $p_A$, $q_A$ a pair of basic functions such that $q_A(x)\neq0$ for
  every $x$ in $A$.
\end{lemma}

\begin{proof}
By Corollary~\ref{co:pd-opt}, the set $S=\{(x,t)\in K^m\times K\tq |t|=|f(x)|\}$ is a boolean
combination of $2$\--basic subsets of $K^{m+1}$. The proof of Denef's
Theorem~6.3 in \cite{dene-1984} then applies word-for-word, with basic
functions instead of polynomial functions. It gives a partition of $X$
in finitely many definable pieces $A$, on each of which $|f|^e=|p_A/q_A|$
for some $1$\--basic functions such that $q_A(x)\neq0$ for every $x$ in
$A$. 
\end{proof}

Note that, in the above proof, if $S$ is a boolean combination of
$(d+1)$\--basic sets then Denef's proof of Theorem 6.3 also goes
through and the resulting functions $p_A$, $q_A$ are $d$\--basic. In
particular, it is not sufficient to know that $S$ is a boolean
combination of basic sets (as it would follow directly from
$p$\--optimality), because Denef's argument then would yield functions
$p_A$, $q_A$ which are only $0$\--basic, that is just definable,
without providing any gain. So, contrary to what happened in
Section~\ref{se:cell-dec} with the Cell Decomposition, the
generalization of Denef's Cell Preparation to $p$\--optimal fields is
not at all straightforward: all the results of the previous section
leading to Corollary~\ref{co:pd-opt} seem to be mandatory here, in
order to ensure that $S$ is a boolean combination of $2$\--basic sets.

\begin{remark}\label{re:norm-n}
  Given an integer $n_0\geq1$, the set $1+\pi^{n_0}R$ is a definable
  subgroup of $R^\times$ with finite index. Thus in Lemma~\ref{le:norm} we
  can always assume, refining if necessary the partition of $X$
  (but keeping the same integer $e$ independently of $n_0$), that for
  every $x$ in $A$
  \begin{displaymath}
    f(x)^e=\cU_{n_0}(x)\frac{p_A(x)}{q_A(x)}.
  \end{displaymath}
\end{remark}

\begin{theorem}[Cell preparation]\label{th:cell-prep}
  Assume that $K$ is $p$\--optimal and satisfies the Extreme Value
  Property. Let $(\theta_i:A_i\subseteq K^{m+1}\to K)_{i\in I}$ be a finite family of
  definable functions. Then there exists an integer $e\geq1$ and, for
  every $n\in\NN^*$, a pair of integers $M$, $N$ and a finite family $\cH$
  of presented cells mod $Q_{N,M}^*$ such that $M>2v(e)$, $e$ divides
  $N$, $\cH$ refines $(A_i)_{i\in I}$, and for every $(x,t)\in H$, 
  \begin{equation}\label{eq:cell-prep}
    \theta_i(x,t) = \cU_{e,n}(x,t) h(x)
    \big[\lambda_H^{-1}\big(t-c_H(x)\big)\big]^\frac{\alpha}{e}
  \end{equation}
  for every $i\in I$ and every $H\in \cH$ contained in $A_i$, with
  $h:\widehat{H}\to K$ a continuous definable function and $\alpha\in\ZZ$ (both
  depending on $i$ and $H$)\footnote{If $H$ is of type $0$ then it is
  understood that $\alpha=0$ and we use the conventions that in this case
  $\lambda_H^{-1}=0$ and $0^0=1$.}\!. 
\end{theorem}

\begin{remark}
  Remark~\ref{re:piece-cont} applies to the above theorem as well,
  so the center and boundaries of every cell in $\cH$ can be chosen to be
  continuous.
\end{remark}

\begin{proof}
For each $i$ let $e_i$ be an integer, $\cA_i$ a partition of $A_i$ and
$\cF_i$ a family of basic functions, all given by Lemma~\ref{le:norm}
applied to $\theta_i$. By replacing each $e_i$ with a common
multiple\footnote{Note that we can require $e$ to be divisible as well
by any given integer $N_0$ if needed.} we can assume that all of them
are equal to some integer $e\geq1$. Given an integer $n\geq1$ from
  the theorem, we set $n_0=n+v(e)$ and we refine the partition
$\cA_i$ as in Remark~\ref{re:norm-n}. 

Let $\cA$ be a finite family of definable sets refining $\bigcup_{i\in
I}\cA_i$. We can assume that each of them is a boolean combination of
basic sets of the same power $N$, with $N$ a multiple of $e$. For
every $A$ in $\cA$, every $i\in I$ such that $A_i$ contains $A$ and
every $(x,t)$ in $A$ we have
\begin{equation}\label{eq:cell-dec-theta}
  \theta_i(x,t)^e=\cU_{n_0}(x,t)\frac{p_{i,A}(x,t)}{q_{i,A}(x,t)} 
\end{equation}
with $p_{i,A}$ and $q_{i,A}$ a pair of basic functions such that
$q_{i,A}(x,t)\neq0$ on $A$. 

For each $A$ in $\cA$ let $\cF_A$ be the set of basic functions
involved in a description of $A$ as a boolean combination of basic
sets of power $N$. Theorem~\ref{th:cell-dec} applies to the family
$\cF$ of all the basic functions $p_{i,A}$, $q_{i,A}$ and the
functions in $\cF_A$, for all $A$'s and $i$'s. It gives a
partition of $K^{m+1}$ into finitely many presented cells $B$ mod $K^*$
such that for every $f$ in $\cF$ and every $(x,t)$ in $B$
\begin{equation}\label{eq:cell-dec-f}
  f(x,t)=\cU_{M}(x,t)h_{f,B}(x)\big(t-c_B(x)\big)^{\beta_{f,B}}
\end{equation}
with $M=n_0+2v(N)$, $h_{f,B}:\widehat{B}\to K$ a definable function and
$\beta_{f,B}$ a positive integer. 

Partitioning $\widehat{B}$ if necessary, we can assume that the cosets
$h_{f,B}(x)Q_{N,M}^*$ are constant on $\widehat{B}$. Since 
  $\cU_M(x,t)\in 1+\pi^{M}R\subseteq Q_{N,M}^*$,  by (\ref{eq:cell-dec-f})
  $f(x,t)Q_{N,M}^*$ only depends on $(t-c_B(x))Q_{N,M}^*$.
Hence $B$ can be partitioned into cells $H$ mod $Q_{N,M}^*$ such that
$\widehat{H}=\widehat{B}$, $c_H=c_B$ and $f(x,t)Q_{N,M}^*$ is constant
on $H$, for every $f$ in $\cF$. {\it A fortiori}\/\footnote{Recall
  that $M=n_0+2v(M)>2v(M)$ hence $Q_{N,M}\subseteq P_N$ by Hensel's Lemma.}
$f(x,t)P_N^*$ is constant on $H$ for every $f$ in $\cF$, hence each
$A$ in $\cA$ either contains $H$ or is disjoint from $H$, for every
$A$ in $\cA$. So the family $\cH$ of all those cells $H$ that are
contained in $\bigcup\cA$ refines $\cA$, hence refines $\{A_i\tq i\in I\}$ as
well. 

For every cell $H$ in $\cH$ there is a unique cell $B$ as above
containing $H$. For every $i\in I$ such that $H$ is contained in $A_i$,
the unique $A$ in $\cA$ containing $B$ is also contained in $A_i$. By
(\ref{eq:cell-dec-f}) applied to $f=p_{i,A}$ and to $f=q_{i,A}$, and by
(\ref{eq:cell-dec-theta}) we have for every $(x,t)\in H$
\begin{equation}\label{eq:cell-dec-frac}
  \theta_i(x,t)^e=\cU_{n_0}(x,t)
  \frac{\cU_{M}(x,t)h_{p_{i,A},B}(x)\big(t-c_B(x)\big)^{\beta_{p_{i,A},B}}}
       {\cU_{M}(x,t)h_{q_{i,A},B}(x)\big(t-c_B(x)\big)^{\beta_{q_{i,A},B}}}
\end{equation}
The $\cU_{n_0}$ and $\cU_M$ factors simplify in a single $\cU_{n_0}$ since
$M\geq n_0$. By construction $c_H=c_B$ and $\widehat{H}=\widehat{B}$. So,
for every $(x,t)$ in $H$ we get 
\begin{equation}\label{eq:cell-dec-lamda}
  \theta_i(x,t)^e=\cU_{n_0}(x,t)g(x)\big[\lambda_H^{-1}\big(t-c_H(x)\big)\big]^{\alpha}
\end{equation}
with $g:\widehat{H}\to K$ a definable function and $\alpha\in\ZZ$ (both depending
on $i$ and $H$). Since $n_0>2v(e)$, $(\cU_{n_0}(x,t))^{\frac{1}{e}}$
is well defined and takes values in $1+\pi^{n_0-v(e)}$ by
Lemma~\ref{le:Hensel-DP}, that is $\cU_{n_0}=\cU_{n_0-v(e)}^e$. We
have $n_0-v(e)=n+v(e)\geq n$, hence {\it a fortiori}\/
$\cU_{n_0}=\cU_n^e$. So (\ref{eq:cell-dec-lamda}) becomes
\begin{equation}
  \theta_i(x,t)^e = \cU_n(x,t)^e g(x) 
  \Big(\big[\lambda_H^{-1}\big(t-c_B(x)\big)\big]^\frac{\alpha}{e}\Big)^e
\end{equation}
This implies that $g$ takes values in $P_e$, hence $g=h^e$ for some
definable function $h:\widehat{H}\to K$, from which (\ref{eq:cell-prep})
follows. 
\end{proof}

\begin{corollary}\label{co:isom}
  Suppose that $K$ is $p$\--optimal and satisfies the Extreme
  Value Property.
  Let $(\theta_i:A\subseteq K^m\to K)_{i\in I}$ be a finite family of definable
  functions with the same domain. Then for every integer $n\geq1$, there
  exists an integer $e$, a semi-algebraic set $\tilde A\subseteq K^m$ and a
  definable bijection $\varphi:\tilde A\to A$ such that for every $i\in I$ and
  every $x$ in $\tilde A$
  \begin{displaymath}
    \theta_i\circ\varphi(x)=\cU_{e,n}(x)\tilde\theta_i(x)
  \end{displaymath}
  with $\tilde\theta_i:\tilde A\subseteq K^m\to K$ semi-algebraic functions.
\end{corollary}

\begin{proof}
The proof goes by induction on $m$. Let us assume that it has been
proved for some $m\geq0$ (it is trivial for $m=0$) and that a finite
family $(\theta_i)_{i\in I}$ of definable functions is given with domain $A\subseteq
K^{m+1}$. If $A$ is a disjoint union of sets $B$, it suffices to prove
the result for the restrictions of the $\theta_i$'s to $B$. So, for any
given integer $n\geq1$, by Theorem~\ref{th:cell-prep} we are reduced to
the case when $A$ is a presented cell mod $Q_{N,M}^*$ for some $N$,
$M$ such that for some $e_0\geq1$ dividing $N$, $M>2v(e_0)$ and for every
$i\in I$ and every $(x,t)$ in $A$
\begin{equation}\label{eq:isom-cell}
  \theta_i(x,t) = \cU_{e_0,n}(x,t)h_i(x)
  \big[\lambda_A^{-1}\big(t-c_A(x)\big)\big]^\frac{\alpha_i}{e_0}
\end{equation}
with $h_i:\widehat{A}\to K$ a definable function and $\alpha_i\in\ZZ$.

Let $e_1\geq1$ be an integer, $Y\subseteq K^m$ a semi-algebraic set,
$\psi:Y\to\widehat{A}$ a definable bijection, $\tilde f:Y\to K$ a
semi-algebraic function for each $f$ in $\cF$, all of this given by
the induction hypothesis applied to $\cF=\{\mu_A,\nu_A\}\cup \{h_i\}_{i\in I}$.
Let $\tilde A$ be the set of $(y,s)\in Y\times K$ such that
\begin{displaymath}
  |\tilde\nu_A(y)|\leq|s|\leq|\tilde\mu_A(x)| \mbox{ \ and \ } s\in\lambda_AQ_{N,M}^*. 
\end{displaymath}
Then $\varphi:(y,s)\mapsto(\psi(y),c_A(\psi(y))+s)$ defines a bijection from $\tilde
A$ to $A$. For every $i\in I$ and every $(y,s)\in\tilde A$ we have
\begin{displaymath}
  \theta_i\circ\varphi(y,s)=\cU_{e_0,n}(y,s)\cU_{e_1,n}(y,s)\tilde
  h_i(y)(\lambda_A^{-1}s)^\frac{\alpha_i}{e_0} 
\end{displaymath}
The first two factors can be replaced by $\cU_{e,n}$ with $e$ any
common multiple of $e_0$ and $e_1$. Since $\tilde\theta:(y,s)\mapsto \tilde
h_i(y)(\lambda_A^{-1}s)^\frac{\alpha_i}{e_0}$ is a semi-algebraic function on
$\tilde A$ the conclusion follows.
\end{proof}

Theorem~\ref{th:cell-prep} and Corollary~\ref{co:isom} are exactly
analogous to Theorems~2.8 and 3.1 in \cite{cluc-2004}, except that we
obtain a slightly more precise equality of functions mod
$(1+\pi^nR).\UU_e$ instead of equality of their norm (which is the same
as equality of functions mod $R^\times$). Thus all the applications that
are derived from these theorems in \cite{cluc-2004} for the classical
analytic structure remain valid in every $p$\--optimal field which
satisfies the Extreme Value Property, with exactly the same proofs as
in \cite{cluc-2004}. As already mentioned in the introduction
some of these applications, which concern the constructibility of
functions defined by parametric integrals and gives the rationality of
Poincar\'e series attached to definable functions, have already been
generalised to strongly $p$\--minimal fields in \cite{cubi-leen-2015}.
The other main application of Theorems~2.8 and 3.1 in
\cite{cluc-2004} is the classification of subanalytic sets up to
subanalytic bijections (Theorem~3.2 in \cite{cluc-2004}). It is not
known at the moment if it holds true for strongly $p$\--minimal fields.

\begin{theorem}\label{th:isom-dim}
  Assume that $K$ is $p$\--optimal and satisfies the Extreme
  Value Property.
  Then there exists a definable bijection between two infinite 
  definable sets $A\subseteq K^m$ and $B\subseteq K^n$ if and only if they have the
  same dimension.
\end{theorem}

\begin{proof}
If there is a definable bijection (an ``isomorphism'') between $A$ and
$B$ they have the same dimension by Corollary~6.4 in
\cite{hask-macp-1997}. Conversely, if $A$ and $B$ have the same
dimension $d$, then by Corollary~\ref{co:isom} they are isomorphic to
infinite semi-algebraic sets $\tilde A$ and $\tilde B$ respectively,
both of which have dimension $d$, by Corollary~6.4 in
\cite{hask-macp-1997} again. Then $\tilde A$ and $\tilde B$ are
semi-algebraically isomorphic by the main result of \cite{cluc-2001},
hence $A$ and $B$ are isomorphic. 
\end{proof}

\bibliographystyle{alpha}

\end{document}